\documentclass[10pt]{amsart}
\usepackage{graphicx}
\usepackage{amsmath,amsthm}
\usepackage{amsfonts}
\usepackage{amscd}
\usepackage[a4paper]{geometry}
\usepackage{subfigure}

\newtheorem{theorem}{Theorem}[section]
\newtheorem*{propositionsn}{Proposition}
\newtheorem{proposition}{Proposition}[section]
\newtheorem{lemma}{Lemma}[section]
\newtheorem{corollary}{Corollary}[section]
\newtheorem{definition}{Definition}[section]
\newtheorem{remark}{Remark}[section]

\newcommand{\norm}[1]{\left\Vert#1\right\Vert}

\newcommand{\oo}{\infty}

\newtheorem*{main}{Main Theorem}
\newtheorem*{mainn}{Main Theorem'}
\newtheorem*{theoremsn}{Theorem}

\newcommand{\Fc}{\ensuremath{\mathcal{W}^c}}

\newcommand{\Fcs}{\ensuremath{\mathcal{W}^{cs}}}
\newcommand{\Fcu}{\ensuremath{\mathcal{W}^{cu}}}
\newcommand{\F}{\ensuremath{\mathcal{F}}}

\newcommand{\Ws}[1]{\ensuremath{W^s(#1)}}
\newcommand{\Wu}[1]{\ensuremath{W^u(#1)}}
\newcommand{\Wc}[1]{\ensuremath{W^c(#1)}}
\newcommand{\Wcs}[1]{\ensuremath{L^{cs}_{#1}}}
\newcommand{\Wcu}[1]{\ensuremath{L^{cu}_{#1}}}

\newcommand{\Wsl}[2][\epsilon]{\ensuremath{W^s(#2;#1)}}

\newcommand{\Wul}[2][\epsilon]{\ensuremath{W^u(#2;#1)}}

\newcommand{\dz}[1]{d_{\mathcal{C}^0}(#1)}

\begin{document}
\author{Pablo D. Carrasco}
\address{I.M.P.A., Estrada Dona Castorina 110 CEP 22360 420. Rio de Janeiro. Brazil.}
\email{pdcarrasco@gmail.com}

\title{Compact Dynamical Foliations}

\subjclass{37D30,57R30}%
\keywords{Partially Hyperbolic Diffeomorphisms,Compact Foliations, Bad Sets}%

\date{\today}

\begin{abstract}
According to the work of Dennis Sullivan, there exists a smooth flow on the 5-sphere all of whose orbits are periodic although there is no uniform bound on their periods. The question addressed in this article is whether these type of examples can occur in the partially hyperbolic context. That is, if does there exist a  partially hyperbolic diffeomorphism of a compact manifold such that all the leaves of its center foliation are compact but there is no uniform bound for their volumes. We develop tools to attack the previous question and show that it has negative answer provided that all periodic leaves have finite holonomy.
\end{abstract}

\maketitle

\section{Introduction and Main Results}
\subsection{Partial Hyperbolicity and Foliations.}

It is a remarkable fact that  many smooth dynamical systems preserve an additional geometrical structure. Often those structures come in the form of foliations, and they provide of useful tools for studying the underlying dynamics. The famous ``Hopf's method'' used to prove the ergodicity of conservative hyperbolic flows (among other systems) is a good example of the aforementioned principle. 

In this work we will be concerned with invariant foliations related to partially hyperbolic diffeomorphisms. Those are a natural generalization of the well known hyperbolic systems, where besides the expanding and contracting directions one allows some \emph{center} directions with some domination relation with respect the other ones.  The presence of these center directions permits a very general type of structure, which makes the study of partial hyperbolic systems considerably harder than the study of its hyperbolic counterpart. Due to this generality however, partial hyperbolic systems appear naturally in different branches of mathematics, which together with the beauty of their complexity makes their study one of the most active research areas in dynamical systems today. To carry on the discussion we give the precise definition.

\begin{definition}
Let  $M$\ be a  closed compact manifold. A $\mathcal{C}^1$\ diffeomorphism $f: M \rightarrow  M$\ is \emph{partially hyperbolic} if there exist a continuous splitting of the tangent bundle into a Whitney sum of the form  $$TM=E^u\oplus E^c\oplus E^s$$\ where all bundles are $df$-invariant, the bundles $E^s$\ and $E^u$\ have positive dimension, and a Riemannian metric $\norm{\cdot}$ on $M$ with the properties:

\begin{enumerate}
 \item  For all $x\in M$, for all unit vectors $ v^{\sigma }\in\ E^{\sigma }_x$ $(\sigma=s,u,c)$ $$\norm{d_xf(v^s)}<\norm{d_xf(v^c)}<\norm{d_xf(v^u)}.$$
 \item  $\max_{x\in M}\{\norm{d_xf|E^s}\}\leq \lambda<1< \mu=\min_{x\in M} \{\norm{(d_xf|E^u)^{-1}}\} $.
\end{enumerate}
\end{definition}

The bundles $E^s,E^u,E^c$\ are the \emph{stable, unstable} and \emph{center} bundle respectively. We also define the bundles $E^{cs}=E^c\oplus E^s$\ and $E^{cu}=E^u\oplus E^c$, the \emph{center stable} and \emph{center unstable} bundles.  When $E^c=\{0\}$\ the map $f$\ is called \emph{completely hyperbolic} or \emph{Anosov}. The set of partially hyperbolic maps on a manifold $M$ is $\mathcal{C}^1$ open in $Diff^1(M)$. See Theorem 2.15 in \cite{HPS}.

\begin{remark}
In the usual definitions of partial hyperbolicity, the inequality $(1)$ in definition 1.1 holds for some iterate $f^N$, where $N>0$ is independent of $x$.  The fact that we can assume $N=1$ is due to the existence of adapted metrics for partially hyperbolic systems, due to \cite{AdaptedMet}.
\end{remark}

We refer the reader to \cite{PesinLect} and \cite{PartSurv} for a more throughout introduction to partial hyperbolicity, and for the discussion of various examples. As this article refers to the properties of certain foliations, this is a good point to recall the definition. Denote by $I^k=(-1,1)^k$.

\begin{definition}
Let $M$ be a $m$ dimensional manifold. A partition $\mathcal{F}=\{L_x\}_{x\in M}$\ of $M$ is called a $\mathcal{C}^{r,s}$ foliation of codimension $q$ if there exist an open covering $\mathcal{U}=\{U\}$ of $M$ and a family of continuous functions $\{\psi_{U}:(-1,1)^q\rightarrow Emb^r((-1,1)^p,M)\}_U$\ with the following properties:
\begin{enumerate}
\item Each atom $L_x$\ of $\mathcal{F}$ is a $p=m-q$ dimensional immersed submanifold of $M$ of class $\mathcal{C}^r$ which contains $x$: these atoms are called the \emph{leaves} of the foliation. For $U\in\mathcal{U}$\ and $x\in M$, the connected component of $L_x\cap U$ containing $x$  is called a \emph{plaque} through $x$ and is denoted by $P_x$. Note that the concept of plaque depends (in principle) of $\mathcal{U}$.
\item If $x\in U$ then there exist unique $a\in I^q, b\in I^p$ 	such that $\psi_{U}(a)(b)=x$, and furthermore $Im(\phi_U(a))=P_x$.
\item If $U\cap U'\neq \emptyset$, consider the map $\phi_{U',U}:I^q\rightarrow I^q$ defined by
$$x=\psi_{U}(a)(b)=\psi_{U}(a')(b')\Rightarrow \phi_{U',U}(b)=b'.$$
Then $\phi_{U',U}$\ is of class $\mathcal{C}^s$.
\end{enumerate}
The number $s$\ is the \emph{transverse regularity} of the foliation. If $s\geq 1$ the foliation is said to be differentiable. The sets $U$ are called \emph{foliation boxes}.

A subset $A\subset M$ is called saturated if consists of whole leaves.
\end{definition}

This definition is equivalent to the one given by means of \emph{foliation charts}, as presented for example in \cite{FoliationsI}. Sometimes $\mathcal{C}^{r,0}$ foliations are also called \emph{laminations} in the literature. Compare Section $5$ in \cite{HPS} and Section $4$ in \cite{PesinLect}.

The classical \emph{Stable Manifold Theorem}, which we cite below, shows the existence of invariant foliations (meaning that $f$ permutes their corresponding leaves) tangent to the bundles $E^s$ and $E^u$. However, the existence of an (invariant) foliation tangent to $E^c$ cannot be guaranteed in general, as the example in \cite{DiffDyn} shows (see also \cite{Nodyncoh}). Nonetheless, the case when such a foliation exists is of particular interest since not only it provides a mechanism to study the dynamics, but also this is the case for the majority of known examples. Some partial results in this matter can be found in \cite{WeakFol},\cite{Dyncoher} and \cite{Tranph}. 

\begin{theorem}[Stable Manifold Theorem]
If $f\in Diff^r(M)$\ is partially hyperbolic then there exist $\mathcal{C}^{r,0}$ foliations $\mathcal{W}^s=\{W^s(x)\}_{x\in M},\mathcal{W}^u=\{W^u(x)\}_{x\in M}$\ tangent to $E^s,E^u$ respectively called the \emph{stable} and the \emph{unstable} foliations. The leaves of these foliations are homeomorphic to Euclidean spaces of the corresponding dimension.  
\end{theorem}

See \cite{HPS}, Theorem $4.1$. We point out that the foliations $\mathcal{W}^s,\mathcal{W}^s$\ are seldom differentiable (\cite{AnosovThesis}, pag. 201); their transverse regularity is only H\"older in general \cite{HolFol}.

In this paper we will discuss the case where the center bundle is tangent to an invariant foliation $\mathcal{W}^c$ such that all its leaves are compact; in other words, $\mathcal{W}^c$ is a \emph{compact foliation}.  The prototypical example in this situation is the well known (and extensively researched\footnote{For more information on skew products see for example: \cite{StErgSkew}, \cite{StErgSkew2}, \cite{Tranph}, \cite{TopGroupExtC}, \cite{StErgSkewExt} and \cite{Patho}. }) class of \emph{skew products} and their perturbations. In those cases the structure of the center foliation is simple: it is a trivial fibration by compact leaves. One is led to ask about the similarities and differences of a  general partial hyperbolic system with compact center foliation with respect to skew products, and in particular whether the center foliation has a simple form.

\subsection{Main Result.}

It will be explained in Section $2$ that for a compact foliation (on a compact manifold) the most important property is the existence a uniform upper bound for the Riemannian volume of the leaves.

\begin{definition}
A compact foliation $\mathcal{F}$ on a compact Riemannian manifold $M$\  is called \emph{uniformly compact} if the function $vol:M\rightarrow \mathbb{R}_{+}$\ which assigns
to each point $x$\ the Riemannian volume of the submanifold $L_x\subset M$\ is uniformly bounded from above, i.e.
$$\sup\{vol(L):L\text{ leaf of }\mathcal{F}\}<\infty.
$$
\end{definition}

The striking counterexample of D. Sullivan \cite{CounPer} shows that not every compact foliation is uniformly compact. If the foliation is uniformly compact then its local structure is simple (cf. \cite{FolCpct} and Section $2$), and this permits some hope for classifying them. On the other hand, it is known that for non uniformly compact foliations the geometrical possibilities are very intricate. See \cite{MalFlujo}.

We now state our main theorem. We will work with the natural class of dynamically coherent diffeomorphisms.

\begin{definition}\label{dyncoh}
A partially hyperbolic diffeomorphism is \emph{dynamically coherent} if there exist $f$-invariant $\mathcal{C}^{1,0}$ foliations $\Fc=\{L_x\}_{x\in M},\Fcs=\{L^{cs}_x\}_{x\in M},\Fcu=\{L^{cu}_x\}_{x\in M}$\ tangent to $E^c,E^{cs}$\ and $E^{cu}$\ respectively, and such that for every $x\in M$, $L_x=L^{cs}_x\cap\ L^{cu}_x$.
\end{definition}

The foliations $\Fcs,\Fcu$ are the \emph{center stable} and \emph{center unstable} foliations. We point out that all known examples where the bundle $E^c$\ is integrable are dynamically coherent. See \cite{stableerg}, \cite{Dyncoher} and compare \cite{Nodyncoh}.

\vspace{0.2cm}

\textbf{Convention:} Every partially hyperbolic diffeomorphism considered in this paper will be dynamically coherent.

\vspace{0.2cm}

\begin{main}
Let $f$\ be a partially hyperbolic diffeomorphism whose center foliation \Fc\ is compact. Suppose that every $f$-periodic center leaf $L_x$\ has neighborhoods $A_x\subset L^{cs}_x, B_x\subset L^{cu}_x$ so that $vol|A_x,vol|B_x$ are bounded from above. Then \Fc\ is uniformly compact.
\end{main}

It will explained later that for a compact foliation, the fact that the function $vol$ is locally bounded near a leaf is equivalent to such a leaf having finite holonomy. This is discussed in Section $2$ where, for convenience of the reader, we also recall the concept of holonomy for a general foliation. Hence, the main theorem can be also phrased as:

\begin{mainn}\label{thmA}
Let $f$\ be a partially hyperbolic diffeomorphism whose center foliation \Fc\ is compact. Then \Fc\ is uniformly compact if and only if every $f$-periodic center leaf has finite holonomy.
\end{mainn}

The advantage of the latter formulation is two-fold:  it expresses an analytical fact (boundedness of the function $vol$) in a purely topological-geometrical fashion and moreover, it reveals the relation with the dynamics. The Main Theorem' can be used with the following reduction.

For a center leaf $L_x$, its holonomy group can be studied by considering the holonomy groups of $L_x$ inside $L^{cs}_x$ and $L^{cu}_x$. We will denote by $G^s_x,G^u_x$ the corresponding holonomy groups of $L_x$ when considered as part of the foliations $\Fc|L^{cs}_x,\Fc|L^{cu}_x$ respectively.

\begin{definition}
The groups $G^s_x,G^u_x$ are called the \emph{stable holonomy group} and the \emph{unstable holonomy group} of the leaf $L_x$.
\end{definition}

Then we have the following useful Proposition, which will be proved in the next section (see Proposition \ref{producto}).

\begin{propositionsn}
For every $L_x\in\Fc$, its holonomy group is finite if and only if both groups $G^s_x,G^u_x$ are finite.
\end{propositionsn}

One can represent the stable and  unstable holonomy groups of $L_x$ by germs of maps defined (locally) in the corresponding strong manifold passing through $x$ (either $\Ws{x}$ or $\Wu{x}$). If we assume further that $L_x$ is $f$-periodic, it is natural to conjugate with the dynamics to try to show that these holonomy maps are globally defined on the complete stable or unstable manifold, and not only on a local transversal as is usually the case for a general foliation. This fact, together with the Main Theorem' provides a substantial simplification for the study of compact center foliations.  To illustrate the technique, we consider the concept of \emph{completeness} of the foliation \Fc.

\subsection{Completeness.}

As far as the author knows, the concept of completeness in partially hyperbolic dynamics was first considered by C. Bonatti and A. Wilkinson in \cite{Tranph}. To give the definition, consider (for $f:M \rightarrow M$ partially hyperbolic) a leaf $L_x\in\Fc$\ and define the sets
$$\Ws{L_x}=  \bigcup_{y\in L_x}\Ws{y}$$
$$\Wu{L_x}=\bigcup_{y\in L_x}\Wu{y}.$$

It follows by dynamical coherence that $\Ws{L_x}\subset L^{cs}_x,\Wu{L_x}\subset L^{cu}_x$ are (relatively) open subsets saturated by the corresponding strong foliation. The condition of being saturated by the center foliation however, is much more subtle.

\begin{definition}\label{defcomp}
The submanifolds $\Ws{L_x}$\ and $\Wu{L_x}$\ are said to be \emph{complete} if they are saturated by the center foliation. The center foliation is complete if for every center leaf $L_x$\ the submanifolds $\Ws{L_x}$\ and $\Wu{L_x}$\ are complete.
\end{definition}

Completeness of $\Ws{L_x}$\ and $\Wu{L_x}$ is the same as metric completeness inside $L^{cs}_x,L_x^{cu}$ respectively (cf. Prop. \ref{completo}). We then have the useful criterion to prove uniform compactness of \Fc.

\begin{theorem}\label{thmB}
Let $f$\ be a partially hyperbolic diffeomorphism whose center foliation \Fc\ is compact and complete. Then \Fc\ is uniformly compact. 
\end{theorem}

This is direct consequence of the Main Theorem' together with the following Proposition.

\begin{proposition}\label{bon}
Consider $f$ partially hyperbolic with compact center foliation \Fc. If $L_x$ is an $f$-periodic center leaf such that $\Ws{L_x}$ is complete, then $G^s_x$ is finite.
\end{proposition}

\begin{proof}
As $L_x$ is $f$-periodic, a simple classical observation shows that a strong stable manifold can intersect $L_x$ at most once. By completeness and the previous remark, any center leaf $L_y\subset L^{cs}_x$ is a topological covering of $L_x$, where the projection $P_{L_y}:L_y\rightarrow L_x$ is given by $P_{L_y}(z)=\Ws{z}\cap L_x$. Since $L_x,L_y$ are compact, these covers are in fact finite. Suppose that the loop $[\alpha]\in \pi_1(L,x)$ yields a holonomy element represented by a map $h: D\rightarrow \Ws{x}$, where $D\subset \Ws{x}$ is a disc centered at $x$. For any $y\in \Ws{x}$ the lift of $\alpha$ to $L_y$ only depends on $[\alpha]$, and this readily implies that $h$ can be extended continuously to $\Ws{x}$. Hence, the holonomy group of $L_x$ is represented by a group $G$ of (globally defined) homeomorphisms of $\Ws{x}$, and furthermore the orbits of the action of $G$ on $\Ws{x}$ are finite, as was remarked before. The proof concludes by using the following theorem of D. Montgomery \cite{PoinPer} (see \cite{FolCpct} for the version below).
\end{proof}

\begin{theorem}\label{Montgomery}
Suppose that $G$ is a group acting effectively by homemorphisms on a connected manifold for which every orbit is finite. Then $G$\ is finite. 
\end{theorem}

\begin{remark}
As the holonomy is $L_x$ globally defined, it follows that $\Ws{L_x}$ is the suspension of the holonomy group (see Section $2$). This was observed by C. Bonatti, who also kindly provided me with a simpler proof of Theorem \ref{thmB} than my original one.
\end{remark}

We give below another example of the use of the Main' Theorem.

\subsection{Lyapunov Stability}

The proof of Theorem \ref{thmB} contains a useful technique to study compact center foliations, namely:
\begin{enumerate}
	\item first establish that if $L_x\in\Fc$ is $f$-periodic (or with no loss of generality, fixed) then $G^s_x$ is represented by homeomorphisms of $\Ws{x}$ which are globally defined, and then
	\item prove that for any center leaf $L_y\subset L^{cs}_x$, the number of intersections of $L_y\cap \Ws{x}$ is finite. 
\end{enumerate}

The first part can be attacked with the help of the dynamics. To exemplify, start noting that since $\pi_1(L_x,x)$ is finitely generated, given $s>0$ there exist finitely many homotopy classes $[\alpha]\in \pi_1(L_x,x)$\ with representatives whose length is less than $s$. For convenience define the length of a class $[\alpha]\in \pi_1(L_x,x)$ as
$$
\inf\{\text{length}(\beta):[\beta]=[\alpha], \beta\text{ rectifiable loop}\}.
$$
The previous observation implies that given $s>0$ there exists an open disc $D(x;r_x)\subset \Ws{x}$ of center $x$ and radius $r_x$, and such that if a class of loops $[\alpha]$ has length less than $s$, then it induces a holonomy element represented by a map $h_{[\alpha]}: D(x;r_x)\rightarrow \Ws{x}$. Using compactness of $L_x$, the size of $D(x;r_x)$ can be taken uniform in $L_x$, i.e. for $s>0$ any $y\in L_x$ has a corresponding disc $D(y,r_y)\subset \Ws{y}$ as above and $\inf\{r_y:y\in L_x\}>0$.

To extend $h_{[\alpha]}$ to a larger disc, and since distances along stable manifolds contract, one can conjugate with $f$ and define its extension as
$$
f^{-1}\circ h_{[f\alpha]} \circ f.
$$
The difficulty is that $[f\alpha]$ can have length bigger than $s$, and thus the domain of definition of $h_{[f\alpha]}$ could become significantly smaller. On the other hand, if the length of $\{[f^n\alpha]\}_{n\in \mathbb{N}}$ remains bounded the previous argument yields a global extension of $h_{[\alpha]}$. This is the case, for example, of one dimensional foliations. Another of such cases is the following.

\begin{definition} A partially hyperbolic diffeomorphism $f:M\rightarrow M$ is called Lyapunov stable (in the center direction) if given $\epsilon>0$ there exists $\delta>0$ such that for every piecewise $\mathcal{C}^1$ curve $\alpha$ tangent to $E^c$ such that $\text{length}(\alpha)<\delta$ it holds
$$
\forall n\geq 0, \quad \text{length}(f^n\alpha)<\epsilon.
$$
\end{definition}

It is a result of F. Rodriguez-Hertz, J. Rodriguez-Hertz and R. Ures  (Corollary 7.6 in \cite{PartSurv}: see also Theorem 7.5 in \cite{HPS}) that if $f$ and $f^{-1}$ are Lyapunov stable then the bundles $E^{cs}, E^{cu}$ are integrable. If moreover the bundle $E^c$ is assumed to integrate to an invariant foliation then $f$ is dynamically coherent.

\begin{theorem}\label{Lya}
Let $f:M\rightarrow M$\ be  partially hyperbolic diffeomorphism with invariant compact center foliation \Fc\ and such that $f$ and $f^{-1}$ are Lyapunov stable. Then \Fc\ is uniformly compact.
\end{theorem}

\begin{proof}
Consider a leaf $L_x\in \Fc$ which is fixed by $f$, and a holonomy element represented by a map $h_{[\alpha]}: D(x;r_x)\rightarrow \Ws{x}$ as discussed above, where $\alpha$ is piecewise $\mathcal{C}^1$. Let $\delta>0$ be the number given in the definition of Lyapunov stability which corresponds to $\epsilon=1$ and consider $k>0$ such that $\alpha$ can be partitioned in $k$ curves of length less than $\delta$. Then for every $n\geq 0$ we have that $\text{length}(f^n\alpha)<k$, and as was previously remarked this implies that $h_{[\alpha]}$ can be extended to $\Ws{x}$. It remains to show that  $|L_y\cap \Ws{x}|<\oo$ for any $y\in L^{cs}_x$.

Assume by contradiction that there exist infinitely many points  $(y_i)_i$ in $L_y\cap \Ws{x}$ and take a $\mathcal{A}$ covering of $L_x$ by foliation boxes from \Fc. Let $\epsilon>0$ be much smaller than the sizes of the plaques of $\mathcal{A}$ and consider $\rho$ the number associated to $\epsilon$ given in the the definition of Lyapunov stability. Since $L_y$ is compact there exist  $y_i,y_j\in L_y\cap W$ satisfying $d(y_i,y_j)<\rho$. Join these points by a $\mathcal{C}^1$ curve $\beta$ in $L_y$ and observe that under iterations these points approach $L_x$. But for a high enough iterate the points $f^n(y_i),f^n(y_j)$ are in different plaques of $f^n{L_y}$ implying that the curve $f^n\beta$ has length bigger than $\epsilon$. This is a contradiction by definition of $\rho$.
\end{proof}

Lyapunov stability is satisfied for example when there exists some constant $C>0$ such that
$$
\forall n\geq 0,\ \frac{1}{C}\leq m(df^n|E^c)\leq \norm{df^n|E^c}\leq C.
$$

\begin{definition}
A  partially hyperbolic diffeomorphism $f$\ is called \emph{center isometric} if for every $x\in M$, for every $v\in E^c_x$\
$$\norm{d_xf(v)}=\norm{v}.$$
\end{definition}

In this case $f$\ is dynamically coherent \cite{OnDynCoh}. From the previous observation and Theorem \ref{Lya} we deduce.

\begin{theorem}
Let $f:M\rightarrow M$\ be a center isometric partially hyperbolic diffeomorphism whose center foliation \Fc\ is compact. Then \Fc\ is uniformly compact.
\end{theorem}

The organization of the rest of the article is as follows. In the next section we discuss generalities of compact foliations, and in particular we study the concept of holonomy giving the necessary background for the proof of the main result. This is done in the third section. We finish the article with some remarks about uniformly compact center foliations. An appendix containing a proof of the general version of Reeb's Stability Theorem needed for the arguments used in the third section is also included.

\section{Compact Foliations}

In this section we study some properties of compact foliations which will be needed for the proof of the main result. We start by reviewing the important concept of holonomy together with various equivalent conditions for a compact foliation to be uniformly compact. Then we recall the concept of \emph{Epstein's filtration of Bad Sets}, which plays a central role in the proof of our Main Theorem. In the last two parts we specialize to the case when the foliation is the central foliation of a partially hyperbolic diffeomorphism.

\subsection{Holonomy and Compact Foliations.}

Throughout this part $M$ is $m$-dimensional compact manifold, $\mathcal{F}$ is a codimension $q$ compact foliation of class $\mathcal{C}^{r,0}$ on $M$ and $E \subset M$ is a locally compact saturated set.

Fix a leaf $L\in  \mathcal{F}$. Since the leaves of $\F$ are differentiable, there exist a $q$-disc bundle $p_L:W_{L}\rightarrow L$ such that
$$
\forall x\in L, D(x)=p_L^{-1}(x) \text{ is an open } q\text{-disc transverse to the leaves of } \mathcal{F}.
$$

If $y\in W_L$ then the map $p_L|W\cap L_y\rightarrow L$\ is a local embedding, and thus by continuity if $y$ is sufficiently close to $L$ then $p_L(W\cap L_y)=L$. From this one can deduce the following (see Proposition 7.1 in \cite{FolCpct} or Proposition 4.1 in \cite{FolCpctForm}).

\begin{lemma}\label{coveringleaves}
If $x\in L$ then for every integer $n\geq 1$ there exist an open neighborhood  $V$ of $x$ such that if $y\in V$ then either
\begin{enumerate}
\item $p: V\cap L_y\rightarrow L$  is more than $n$-to one, or
\item $p: V\cap L_y\rightarrow L$  is a $k$-to one covering for some $1\leq k\leq n$.
\end{enumerate}
\end{lemma}

\begin{corollary}[D.B.A. Epstein]\label{volumeEpstein}
If $x\in L$ then for every integer $n\geq 1$, for every $\epsilon>0$ there exist an open neighborhood  $V$ of $x$ such that if $y\in V$ then either
\begin{enumerate}
\item $vol(L_y)>n vol(L)$, or
\item $|vol(L_y)-k\cdot vol(L)|<\epsilon$ for some $1\leq k\leq n$.
\end{enumerate}

In particular $vol|E$ is lower semi-continuous.
\end{corollary}

Consider a continuous loop $\alpha: [0,1]\rightarrow L_x$ such that $\alpha(0)=x=\alpha(1)$. If $y\in D(x)$\ is sufficiently close to $x$, by Lemma \ref{coveringleaves} one can lift $\alpha$ to a path $\widehat{\alpha}_y$ in $L_y$ such that $\widehat{\alpha}_y(0)=y$. This procedure defines a homeomorphism $h: O_h(x)\subset D(x)\rightarrow D(x)$ where $O_h(x)$ is an open disc centered at $x$ by setting $h(y):= \widehat{\alpha}_y(1)$. Standard arguments in Foliation Theory (cf. \cite{FoliationsI}, chap. 2) show that the germ $germ_x(h)$ of $h$ at $x$  only depends on the homotopy class of $\alpha$ in $\pi_1(L_x,x)$.

\begin{definition}
The holonomy group of $L_x$ at $x$ is the group of germs
$$
G_x=\{germ_x(h):h: O_h(x)\subset D(x)\rightarrow D(x)\}
$$

For points $x\in E$, the restricted holonomy group $G_x|E$ is defined similarly using the germs of the maps $h|E$.
\end{definition}
For another (equivalent) definition of the holonomy group see \cite{HolFolRev}.

It can be shown that changing the basepoint $x$ to another point in the same leaf or changing the transversal $D(x)$ to another smooth transversal $T(x)$ lead to conjugated holonomy groups. For this reason sometimes the holonomy group of a leaf is identified with a subgroup of the group of germs at zero of $\mathbb{R}^q$.

\begin{definition}
Let $x\in E$. The leaf $L_x$ is said to have finite holonomy (trivial holonomy) in $E$ if the corresponding holonomy group $G_x|E$ is finite (trivial).
\end{definition}

If a leaf $L_x$ has finite holonomy then its local structure is simple. For a set $S$\ let $|S|\in\mathbb{N}\cup\{\oo\}$\ denote its number of elements.

\begin{theorem}[Reeb's Stability Theorem]\label{reebpotente}
Let $E$ be a locally compact saturated set of a foliated manifold $M$. Let $x\in E$ such that its leaf $L_x$ is compact and assume that $G_x|E$ is finite. Then there exist arbitrarily small relatively open sets $U\subset E$ such that
\begin{enumerate}
\item $U$ is saturated.
\item For every $y\in U$ the map $p|:L_y\rightarrow L_x$ is a covering map with less than equal $|G_x|E|$ number of sheets. In particular for every $y\in U$ the group $G_y|E$ is finite.
\end{enumerate}
\end{theorem}

The version given above is due to C. Ehresmann and W. Shih  \cite{reebpotente}, and its validity in this degree of generality depends on the existence of tubular neighborhoods on the leaf $L_x$. See the Appendix for its proof.

\emph{Local Model - } In the setting of Reeb's Stability Theorem, fix one of the sets $U$ and let $V=D(x)\cap U$. Since $U$ is saturated, the group homomorphism $\pi_1(L_x,x)\rightarrow G_x|E$ that assigns to each loop based at $x$ the germ of the corresponding homeorphism can be split as
$$
\pi_1(L_x,x)\xrightarrow[]{\psi} Homeo_x(V)\rightarrow G_x|E
$$
where $Homeo_x(V)$ is the group of homomorphisms of $V$ that fix $x$. Denote by $\widehat{L_x}$\ the covering space of $L_x$\ corresponding to the subgroup $\ker(\psi)$\ and let $H=Im(\psi)$. Then $H\simeq \pi_1(L,x)/\ker(\psi) $\ is isomorphic to the Deck transformation group of $\widehat{L_x}$,  and thus acts with the product action on the space $\widehat{L_x}\times V$.  This action is free and properly discontinuous, hence the quotient map
$$q: \widehat{L_x}\times V\rightarrow\widehat{L_x}\times_H V:=\widehat{L_x}\times V/\sim$$ is a covering map over the second countable (because $V$ is second countable), locally compact 
Hausdorff space $\widehat{L_x}\times_H V$. Note that by the Tychonoff-Urysohn Metrization Theorem the space  $\widehat{L_x}\times_H V$ is a separable metrizable space.  Even more, if $H$ is equipped with the discrete topology then $q: \widehat{L_x}\times V\rightarrow\widehat{L_x}\times_H V$\ is a $H$-bundle. See for example\footnote{Even though the proof there is stated for differentiable foliations, the arguments adapt easily to our context.} \cite{FoliationsI}, chap. $3$ and compare with Theorem 4.3 in \cite{FolCpct}. Observe that there exists a partition $\mathcal{H}=\{q(\widehat{L_x}\times \{v\})\}_{v\in V}$\ of $\widehat{L_x}\times_H V$\ whose atoms (also denominated \emph{leaves}) are $(m-q)$-dimensional $\mathcal{C}^r$ manifolds. These types of constructions are called \emph{suspensions} and are due to A. Haefliger.

\begin{theorem}[Local Model]\label{localmodel}
In the hypothesis of Reeb's Stability Theorem each set $U$ is homeomorphic to the set of the form $\widehat{L_x}\times_H V$\ by an homeomorphism that sends the leaves of $\mathcal{F}$ to the leaves of $\mathcal{H}$.
\end{theorem}

In particular it follows that a uniformly compact foliation has a nice structure around each leaf. 

It was first proved by D.B.A Epstein \cite{Per3Man} that the set of continuity points of $vol|E$ coincides with the points $x\in E$ with trivial holonomy. Since a lower semi-continuous function is continuous on a residual set, and residual subsets of locally compact spaces are dense we obtain (compare \cite{leavesnohol1}):

\begin{corollary}\label{genericleaves}
Let $E$ be a locally compact saturated set in a foliated manifold $M$. Then the set of points $x\in E$ with trivial holonomy is an open and dense set subset of $E$.
\end{corollary}

\emph{Equivalent definitions of uniform compactness - } The following three Propositions give a useful list of equivalent definitions to uniform compactness. For the proof we refer the reader to \cite{FolCpct}.

\begin{proposition}\label{voleqhol}
Let $E$ be a locally compact saturated and $x\in E$. Then $G_x|E$ is finite if and only if $vol|E$ is locally bounded at $x$ (meaning that there exist a neighborhood $U\subset E$ of $x$ such that $vol|U$ is bounded).

Thus the foliation $\mathcal{F}$ is uniformly compact if and only if for every $x\in M$ the the group $G_x$ is finite.
\end{proposition}

Denote by  $\pi_E: E\rightarrow E/\mathcal{F}$\  the canonical quotient map. The saturation of a subset $F\subset E$ is the union of all the leaves that intersect $F$.

\begin{proposition}\label{posibilidadcompacta}
The following properties are equivalent.

\begin{enumerate}
\item $\pi_E$\ is closed.
\item $E/\mathcal{F}$\ is Hausdorff.
\item Every leaf $L\subset L$\ has arbitrarily small saturated neighborhoods in $E$.
\item For every $K\subset E$\ compact, its saturation is compact.
\end{enumerate}
\end{proposition}

\begin{proposition}\label{unifcompeqhaus}
Assume that for every $x\in E$ the group $G_x|E$ is finite. Then $E/\mathcal{F}$\ is Hausdorff. Conversely, if the set $E$ is a manifold and $E/\mathcal{F}$\ is Hausdorff then every leaf in $E$ has finite holonomy in $E$.
\end{proposition}

One of the aims of this article is to provide further equivalent conditions for a compact foliation to be uniformly compact in the case where the foliation is assumed to be the center foliation of a dynamically coherent partially hyperbolic diffeomorphism. In particular we establish that in this case uniform compactness is equivalent to completeness. See Theorem \ref{thmB} and Proposition \ref{unifimpcomp}.

\subsection{Bad Sets.}\label{Sect. 3}
We now study more carefully the sets of leaves with infinite holonomy, and for that we will use the \emph{Epstein's filtration of Bad Sets} (see \cite{Per3Man}). The definition is as follows.

\begin{definition}
Consider a compact foliation $\mathcal{F}$ on a manifold $M$. The Epstein filtration of bad sets is the family $\{B_{\alpha}\}_{\alpha}$ indexed by the ordinals, where $B_0=\{x\in M: |G_x|=\infty\}$ and for $\alpha>0$\
$$
B_{\alpha}=\begin{cases}
\{x\in B_{\alpha-1}: G_x|B_{\alpha-1}\neq \{0\}\}&\text{ if }\alpha\text{ is a successor ordinal.}\\
\cap_{\beta<\alpha} B_{\beta}&\text{ if }\alpha\text{ is a limit ordinal.}
\end{cases}
$$

The sets $B_{\alpha}$ are called the bad sets of the foliation.
\end{definition}

For the rest of this section we fix a compact foliation $\mathcal{F}$ on a \emph{compact} manifold $M$ with Epstein filtration $\{B_{\alpha}\}_{\alpha}$.  We collect in the next Proposition some facts about the Epstein filtration (c.f. \cite{Per3Man}). Because of the central role that it plays in our arguments, and for convenience of the reader, we give the proof.

\begin{proposition}\label{propertiesBalfa}
The following properties hold.
\begin{enumerate}
\item Each $B_{\alpha}$ is a compact saturated set.
\item If $\beta<\alpha$ then $B_{\beta}\supset B_{\alpha}$ is nowhere dense.
\item There exist a successor ordinal $\widehat{\alpha}$ less than the first uncountable ordinal such that $B_{\widehat{\alpha}-1}\neq \emptyset, B_{\widehat{\alpha}}=\emptyset$. We will write $B_{end}:=B_{\widehat{\alpha}-1}$.
\item For every ordinal $\alpha$,  the set $B_{\alpha}\setminus B_{\alpha+1}$\ is a locally trivial fibration with fibers the leaves of $\mathcal{F}|B_{\alpha}\setminus B_{\alpha+1}$. In particular the set $B_{end}$ is a compact locally trivial fibration with compact fibers.
\end{enumerate}
\end{proposition}

\begin{proof}

The first and the second part follow directly from Corollary \ref{genericleaves} (we emphasize that $M$ is compact otherwise we could only assert that $B_{\alpha}$ is closed).

For the third part, assume for the sake of contradiction that there exists an uncountable ordinal such that the corresponding bad set is non-empty. Let $\Omega$ be the first uncountable ordinal, and for each $\alpha<\Omega$ choose a point $x_{\alpha}\in B_{\alpha}\setminus B_{\alpha+1}$. The set $X=\{x_{\alpha}\}_{\alpha<\Omega}$ is uncountable, and since the manifold $M$ has a countable basis it follows that there exist some $\alpha_0$ such that $x_{\alpha_0}$ is an accumulation point of $X$, and moreover each neighborhood of $x_{\alpha_0}$ has uncountably many points of $X$ (see for example \cite{Kelley} chap. 1). There exist countably many ordinals smaller that $\alpha_0$ so in particular each neighborhood of $x_{\alpha_0}$ contains points $x_{\alpha}$ with $\alpha>\alpha_0$. This is a contradiction because the set $B_{\alpha_0}\setminus B_{\alpha_0+1}$ is open in $B_{\alpha_0}$. Observe that since the intersection of nested non-empty compact sets is non-empty, $\widehat{\alpha}$ cannot be a limit ordinal.

The last part is a consequence of Theorem \ref{localmodel}.
\end{proof}

On the other hand, the next Lemma shows that $B_{\alpha+1}$ is precisely the obstruction for $\mathcal{F}|B_{\alpha}$ to be a locally trivial fibration. Compare Section 1 in \cite{BadSets2}.

\begin{lemma}\label{notfibration}
Consider a bad set $B_{\alpha}$, a point $x\in B_{\alpha+1}$ and define $X_{\alpha}=B_{\alpha}\setminus B_{\alpha+1}\cup \{L_x\}$. Then $\mathcal{F}|X_{\alpha}$ is not a locally trivial fibration.
\end{lemma}

\begin{proof}
Since $x\in B_{\alpha+1}$ its holonomy group $G_x|B_{\alpha}$ is not trivial. Consider the disc bundle $p:W\rightarrow L_x$ as discussed in the previous section, and note that  in a fibration the volume is essentially locally constant. Hence, it suffices to show the following.

\emph{Claim:} In each neighborhood $U\subset B_{\alpha}$ of $x$ there exist points $y\in U\cap B_{\alpha}\setminus B_{\alpha+1}$ such that $p|L_y\rightarrow L_x$ is more than $1$ to $1$.

Fix one of such neighborhoods $U$ and observe by Lemma \ref{coveringleaves} that one can find some neighborhood $V\subset U$ with the property that for all $y\in V, p:L_y\rightarrow L_x$ is $k$ to one with $k\geq 1$. Since $G_x|B_{\alpha}\neq \{0\}$, there exists some $y_0\in V$ such that the projection is $k$ to $1$ with $k>1$. If $y_0\ni B_{\alpha}$ we are done. Otherwise, by Corollary \ref{genericleaves}, arbitrarily close to $y_0$ there exist points $y\in B_{\alpha}\setminus B_{\alpha+1}$. Observe then that for those points, their leaves have to intersect $p^{-1}(x)$ at least as many times as $L_{y_0}$ does, hence $p:L_y\rightarrow L_x$ is also more than $1$ to $1$ and we finish the proof.
\end{proof}

\begin{remark}
It would be interesting to show that $X_{\alpha}$ is never locally compact, or what it is equivalent, that $x$ has a neighborhood $U\subset B_{\alpha}$ such that $U\cap B_{\alpha+1}=L_x$. That is probably the case, but I do not know any argument to prove it.
\end{remark}

\subsection{Holonomy of Compact Dynamical Foliations}\label{Sect. 4}

Let us go back to dynamics. Fix $f:M\rightarrow M$ partially hyperbolic with compact center foliation \Fc. Given a point $x\in M$\ and a positive number $\gamma>0$\ we will denote by $\Wsl[\gamma]{x}$\ the open disc of size $\gamma$\ inside the leaf $\Ws{x}$, measured with the intrinsic metric. Similarly for $\Wul[\gamma]{x},L(x;\gamma)$. For a leaf $L\in \Fc$ we also define 

$$\Wsl[\gamma]{L}=\bigcup_{x\in L}\Wsl[\gamma]{x}$$
$$\Wul[\gamma]{L}=\bigcup_{x\in L}\Wul[\gamma]{x}.$$

\emph{Coordinates on transverse discs - } In the definition or partial hyperbolicity the metric used can be assumed to make the bundles $E^c,E^s$\ and $E^u$\ mutually orthogonal \cite{AdaptedMet}. From dynamical coherence then it follows that we have \emph{local product structure}: there exists some $r>0$\ such that if we denote by $H_x^r=\Wul[2r]{L(x;2r)},V_x^r=\Wsl[2r]{L(x;2r)}$\ then $d(x,y)<r$\ implies that $H_x^r\cap V_y^r$ contains a center plaque of diameter bigger than or equal to $r$. See Section 7 of \cite{HPS}.

In particular if $D$\ is a small $q$-dimensional transverse disc to \Fc\ centered at $x$\ and $y\in D$\ there exist uniquely defined points $y^u_D\in \Wul[2r]{x},y^s_D\in \Wsl[2r]{x}$\ such that $y=D\cap V_{y^u_D}^r\cap H_{y^s_D}^r$. The map  $\Psi_x^D:D\rightarrow \Wul[2r]{x}\times\Wsl[2r]{x}$\ given by
$$
\Psi_x^D(y)=(y^u_D,y^s_D)
$$
is an open embedding, and thus define a continuous system of coordinates on $D$ (see figure \ref{coordenadas}).

\begin{figure}
	\centering
		\includegraphics[width=8cm]{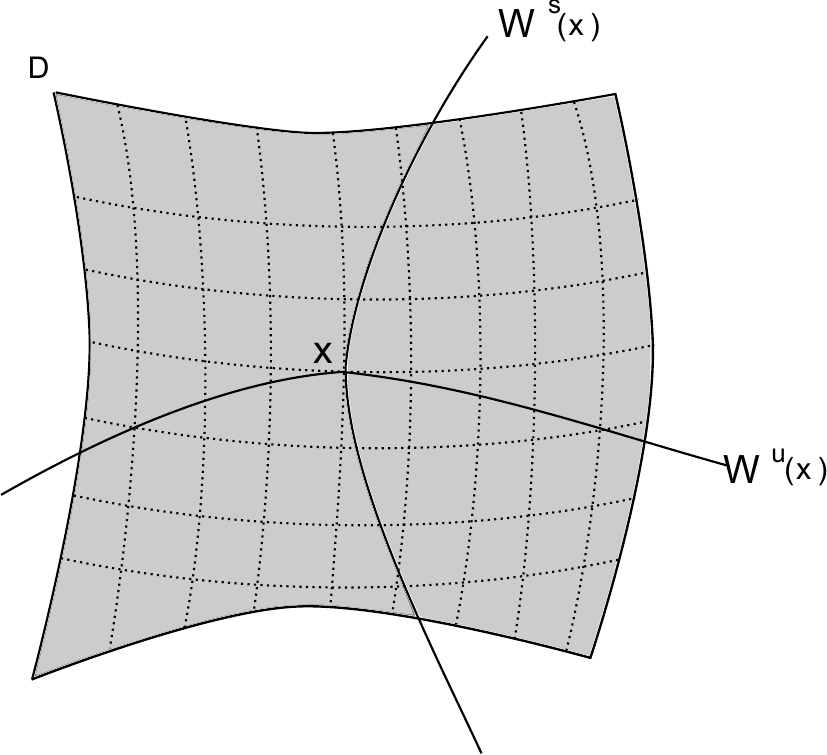}
	\caption{Transverse coordinates in $D$.}
	\label{coordenadas}
\end{figure}


We then have the following useful Proposition.

\begin{proposition}\label{producto}\
For every $x\in M$ it holds
$$\max\{|G^s_x|,|G^u_x|\}\leq |G_x|\leq |G^u_x||G^s_x|.$$
In particular if both groups $G^u_x,G^s_x$ are finite, then $G_x$ is finite.
\end{proposition}

\begin{proof}
Note first that if either $G^s_x$\ or $G^u_x$\ are infinite then by Proposition \ref{voleqhol} there exist center leaves close to $L_x$\ with arbitrarily large volume, so $G_x$\ cannot be finite. Assume now that both groups $G^s_x, G^u_x$ are finite. Observe that the first inequality follows easily by definition of holonomy. For the second one, by Theorem \ref{posibilidadcompacta} the leaf $L_x$\ has saturated neighborhoods $A\subset \Wul[r/2]{L}, B\subset \Wsl[r/2]{L}$. Take now a small $q$-dimensional disc $D$\ centered at $x$  transverse to $\Fc$, and note that any other leaf $L'$\ such that $L'\cap D$\ is sufficiently close to $x$ satisfies
$$
sup\{d(y,L_x):y\in L'\}<r/2
$$
Hence $L'$ is contained in the neighborhood of $L_x$ where the local product structure is defined. We deduce that there exist an open neighborhood $U$\ of $L$\ and  well defined projections $pr_u:U\rightarrow A, pr_s:U\rightarrow B$\ such that for each center leaf $L'\subset U$ the sets $pr_u(L'),pr_s(L')$\ consist of a unique leaf. By using the coordinate system on $D$\ defined above, we obtain at that $G_x$\ is finite and it is generated by the elements of the product $G^u_x\times G^s_x$.
\end{proof}

\begin{remark}
We observe that it does not follow that $G_x=G^s_x\times G^u_x$. As an example, take $M$ the mapping torus of $-Id$ on $\mathbb{T}^3$ and let $f:M\rightarrow M$ be the map induced by $A\times Id : \mathbb{T}^2\times [0,1]$, where $A$ is the usual Thom's map. One readily verifies that $f$ is partially hyperbolic, and that if $L_0$ denotes the center leaf obtained by projecting $\{0\}\times [0,1]$ on $M$ then
$$
G_0\approx G^s_0\approx G^u_0\approx\mathbb{Z}_2.
$$ 
\end{remark}

\begin{corollary}\label{codim1}
Assume that $E^c$\ has codimension one inside both $E^{cs}$\ and $E^{cu}$, then.
\begin{enumerate}
\item  The foliation \Fc\ is uniformly compact. Moreover, if both $E^c$ and its normal bundle $E^s\oplus E^u$ are orientable  then all center leaves have trivial holonomy.
\item There exist a finite covering $\widehat{f}$\ of $f$\ such that $\widehat{f}$\ fibers over a map $g:\mathbb{T}^2\rightarrow \mathbb{T}^2$\ which is conjugate to a hyperbolic automorphism.
\end{enumerate}
\end{corollary}

For the case $\dim M =3$\ this Theorem was noted in see \cite{Tranph} (Proposition 2.12). Uniform boundedness (for general $\mathcal{C}^{1,0}$ compact foliations) also follows by combining the results of \cite{FolCpctForm} and \cite{PointPerHomeo}, although the arguments are much more sophisticated in the general case.

\begin{proof}
We use the following result of A. Haefliger (Theorem 3.2 in \cite{VarFeu}).

\begin{theorem}
Let $\F$ be a $\mathcal{C}^{1,0}$ compact codimension one foliation on a (not necessarily compact) manifold $V$ that is tangent to a continuous sub-bundle of $TV$. Then the saturation of any compact set is compact.
\end{theorem}
By Prop. \ref{unifcompeqhaus} it follows that in this case each leaf of $\F$ has finite holonomy. Furthermore, since the holonomy maps are represented by local homeomorphisms of $\mathbb{R}$, one can show that every holonomy group has order at most two, and in fact has order two precisely when it contains an element which reverses the orientation of $\mathbb{R}$.

Take $x\in M$ and apply Haefliger's Theorem to the foliations $\Fc|\Wcs{x}$ and $\Fc|\Wcu{x}$ to conclude that $G_x^s,G^u_x $ are finite. Part (1) of the statement follows then by Proposition \ref{producto}. Part (2) follows exactly as in  Section 2.4 of \cite{Tranph}.
\end{proof}

\subsection{More on completeness.}

We continue working with a fixed partially hyperbolic diffeomorphism $f: M\rightarrow M$ whose center foliation \Fc\ is compact. In the Introduction we discussed the concept of completeness and we related it with uniform compactness (Theorem \ref{thmB}). Here we investigate some more of its properties. We start by noting the following.

\begin{proposition}\label{completo}
If $\Ws{L_x}$\ is complete. Then $\Ws{L_x}=L^{cs}_x$. 
\end{proposition}

\begin{proof}
It suffices to show that $\Ws{L_x}$\ is closed inside the center stable manifold where it is contained. Take a sequence $(z_n)_n$\ in $\Ws{L_x}$ converging to a point $z\in \Wcs{x}$. Consider a foliation box $U$ of \Fc\ around $z$ and denote by $P_z$ corresponding plaque.  For $n$ sufficiently large, $z_n\in U$ and hence there exist a stable manifold of a point in $L_x$\ that intersects $P_z$. Since it was assumed that $\Ws{L_x}$\ is complete, we conclude that $L_z\subset \Ws{L_x}$, and in particular $z\in\Ws{L_x}$. Thus $\Ws{L_x}$\ is closed.
\end{proof}

\begin{remark}
In the case of 3-manifolds the previous Proposition was proved in \cite{Tranph}.
\end{remark}

We also observe that in our setting (compact center foliations) completeness essentially means that ``leaves do not escape to infinity''. We make precise this idea with the following Lemma (whose proof is immediate).

\begin{lemma}
Suppose that $L,L'$\ are compact center leaves with $L'\subset \Ws{L}$, and assume that $\Ws{L'}$\ is complete. Then $L\subset \Ws{L'}$, and in particular for every $x\in L$\ we have $\Ws{x}\cap L'\neq\emptyset$.
\end{lemma}

We finish by noting that Theorem \ref{thmB} has a converse.

\begin{proposition}\label{unifimpcomp}
Consider $f:M\rightarrow M$ partially hyperbolic with uniformly compact center foliation \Fc. Then \Fc\ is complete.
\end{proposition}

\begin{proof}
As \Fc\ is uniformly compact, by Theorem \ref{posibilidadcompacta} the space $X:=M/\Fc$ is a compact Hausdorff space, thus it is metrizable (cf. Prop. 17 in \cite{gentop}). By a harmless use of notation we will use the same notation for points of $X$ and leaves. A  compatible metric is given by

$$L,L'\in X,\quad d_X(L,L')=\inf\{d(x,y):x\in L, y\in L'\}.$$

Fix a leaf $L\in\Fc$, and consider any other leaf $L'$\ such that $L'\cap\Ws{L}\neq\emptyset$. 

\emph{Claim:} If $\epsilon>0$\ is sufficiently small then there exists some positive integer $N$\ such that 
$$
n\geq N \Rightarrow d_X(f^nL,f^nL')<\epsilon.
$$
For $F\in X$ and $\delta>0$ denote by $B_{X}(F,\delta)$ the open ball in $X$ of center $F$ and radius $\delta$. By compactness of $X$, there exist leaves $F_1,\ldots F_k$ and $0<\delta<\frac{\epsilon}{2}$\ such that
$$
X=\cup_{i=1}^k B_{X}(F_i,\delta).
$$
Now there exist points in $L,L'$ in the same strong stable manifold, and thus by iterating we can guarantee that there exists $N$ so that for $n\geq N$  the leaves $f^nL, f^nL'$ have points whose distance apart is less than the Lebesgue number of the covering $\{ B_{X}(F_i,\delta)\}_{i=1}^n$. Thus
$$\forall n\geq N,\quad d_X(f^nL,f^nL')<2\delta<\epsilon.$$

For $\epsilon>0$ sufficiently small, by Theorem 6.1 in \cite{HPS}  we conclude that $f^nL'\subset\Wsl[2\epsilon]{f^nL}$, and thus $L'\subset\Ws{L}$.
\end{proof}

\begin{corollary}
Let $f$\ as in the previous Proposition and assume that for every periodic center leaf $L$ the submanifolds $\Ws{L}$\ and $\Wu{L}$\ are complete. Then the same is true for every center leaf (i.e. the center foliation is complete).
\end{corollary}

\begin{proof}
The hypothesis imply, by Proposition \ref{bon} and the Main Theorem', that \Fc\ is uniformly compact. By the Previous proposition \Fc\ is complete.
\end{proof}

\begin{remark}
The hypothesis of dynamical coherence in Proposition \ref{unifimpcomp} is superfluous: uniform compactness imply dynamical coherence. This was proved in \cite{CompDynFol}, and also obtained using different methods by C. Pugh \cite{CompimplyDynCoh} and by C. Bonatti-D. Bohnet \cite{CompCenFolFinHol}.
\end{remark}

\subsection{Historical note.}  The history of the problem of deciding whether every compact foliation is uniformly compact can be traced back to G. Reeb who gave in his thesis an example of a flow on a non-compact manifold whose orbits were all periodic, but the time of return (which is proportional to the length of each leaf) was not locally bounded. This example led A. Haefliger to ask whether such type of behaviour could appear in a compact manifold,  or equivalently, if there could be an example of a compact manifold with a compact foliation having locally unbounded volume. Later D.B.A. Epstein proved, using a very sophisticated argument,  that in a compact three manifold this phenomenon could not happen (see \cite{Per3Man}). However in 1976 D. Sullivan gave an example of a compact flow in $S^5$\ where the time of return was not bounded (\cite{CounPer}), and a similar type of example was given by D.B.A. Epstein and E. Vogt in a manifold of dimension 4 (\cite{CounPer3}). As for the case of compact center foliations, recently A. Gogolev proved that such foliations are uniformly compact provided that the manifold has dimension less than six \cite{gogo}.

\section{Proof of the Main Theorem.}

Fix $f:M\rightarrow M$ partially hyperbolic diffeomorphism with compact center foliation $\Fc$. We are going to show that \Fc\ is uniformly compact provided that the holonomy of every $f$-periodic leaf is finite, thus establishing the Main Theorem. I would like to offer my thanks to R. Ures for his aid in the proof of this result, and to the referee for pointing me out inaccuracies in previous versions.

The proof will be achieved by using a variation of R. Bowen's construction of shadowing (Prop. 3.6 in \cite{EquSta}) to show that if $B_{0}$ is not empty then $B_{0}$ contains a periodic leaf. We then assume that 
$B_{0}\neq \emptyset$ and look at the set $B_{end}\subset B_{0}$ discussed in Section \ref{Sect. 3}. We remind the reader that this set is compact, non-empty, $f$-invariant and $\Fc|B_{end}$ is a locally trivial fibration.

Consider a differentiable bundle $N$ which is almost orthogonal to $E^c$  and such that $TM=N\oplus E^c$  (i.e. a differentiable approximation to $E^s\oplus E^u$). For $x\in M$ and $\alpha$ smaller than the injectivity radius of the exponential denote by $N_{\alpha}(x):=exp_x(\{v\in T_xN:\norm{v}\leq \alpha\})$.

\begin{lemma}\label{discoalfa}
There exists $\alpha>0$ such that if $x\in B_{end}$ then $L_x\cap N_{\alpha}(x)=\{x\}.$
\end{lemma}

\begin{proof}
For a given $x\in B_{end}$ compactness of $L_x$ implies the existence of $\alpha(x)$ satisfying the claim.
Since $B_{end}$ is locally trivial there exists $U\subset B_{end}$ saturated neighborhood of $x$ such that for every $y\in U,\ N_{\alpha(x)}(y)\cap L_y=\{y\}$.  Finally, the uniform $\alpha$ can be found by compactness of $B_{end}$. 
\end{proof}

To carry this proof we need some more preparations. For $r\geq 0$ and a manifold $N$ (equipped with a metric $\norm{\cdot}_N$) we denote by $d_{\mathcal{C}^r}$ the $\mathcal{C}^r$ distance in $Emb^r(N,M)$, the set of  $\mathcal{C}^r$ embeddings from $N$ to $M$.

\begin{definition}
Let $\epsilon$ be a  positive number. Two (embedded) submanifolds $N_1,N_2$ of $M$ are said to be $\mathcal{C}^r-\epsilon$ close if there exists a Riemannian manifold $(N,\norm{\cdot}_N)$ and two $\mathcal{C}^r$ embeddings $h_i: N\rightarrow M,\ i=1,2$ such that:
\begin{enumerate}
\item $h_1(N)=N_1,h_2(N)=N_2$.
\item $h_1^{\ast}\norm{\cdot}_M=h_2^{\ast}\norm{\cdot}_M=\norm{\cdot}_N$.
\item $d_{\mathcal{C}^r}(h_1,h_2)<\epsilon$.
\end{enumerate}
By a harmless abuse of language we will write $d_{\mathcal{C}^r}(N_1,N_2)<\epsilon$ as a shorthand for saying that the submanifolds $N_1,N_2$ are $\mathcal{C}^r-\epsilon$ close.
\end{definition}

\begin{remark}
If $N$ is compact the second condition is unnecessary as long as the metric in $N$ is fixed. 
\end{remark}

Proximity in this sense is very strong: we are requiring not only the points of the manifolds $N_1,N_2$ to be close inside $M$, but their parametrizations to be ($\mathcal{C}^r$) close. For example, if $d_{\mathcal{C}^0}(N_1,N_2)<\epsilon$ then their Hausdorff distance is less than $2\epsilon$, i.e. they are close in the Hausdorff distance. But the converse is false: consider the foliation by horizontal circles of the M\"{o}bius band and note that there exist circles $C$ of length $2$ arbitrary close in the Hausdorff distance to the central circle $C_0$ of length one (see figure \ref{Mobius}). However, $d_{\mathcal{C}^0}(C,C_0)\geq 1$ if $C\neq C_0$ (see figure \ref{Mobius}).

\begin{figure}[ht]
  \begin{center}
  \includegraphics[width=7cm]{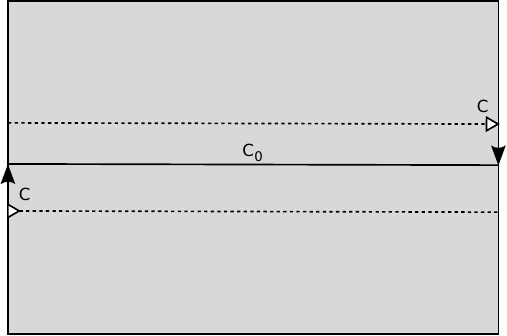}\\
  \caption{Horizontal foliation of the M\"{o}bius band}\label{Mobius}
\end{center}
\end{figure}

The presence of this phenomena  is explained by the fact that $C_0$ has non trivial holonomy. Although some of the results here can be presented in greater generality we content ourselves to the restricted case that we are studying.

For $\delta>0$ denote  

$$
\Gamma(\delta):=\{x\in M:\exists\ y\in B_{end}\text{ s.t. } d_{\mathcal{C}^0}(L_x,L_y)<\delta\}.
$$

\begin{lemma}\label{unicainter}
There exist numbers $\alpha,\delta,\epsilon>0$ such that 
\begin{enumerate}
\item If $x\in \Gamma(\delta)$ then $L_x\cap N_{\alpha}(x) = \{x\}$.
\item If $x,y\in \Gamma(\delta)$ and $d_{\mathcal{C}^0}(L_x,L_y)<\epsilon$ then for every $z\in L_x$ the intersection  $N_{\alpha}(z)\cap L_y$ consists of a unique point.  
\end{enumerate}
\end{lemma}
    
\begin{proof}
Let $\alpha$ be the number given in Lemma \ref{discoalfa}. By Corollary \ref{volumeEpstein} and compactness of $B_{end}$ there exists $0<\delta<\frac{\alpha}{3}$ such that if $d_{\mathcal{C}^0}(L_x,L_y)<\delta$ with $y\in B_{end}$ then
$$|vol(L_x)-vol(L_y)|<\frac{\min{vol|B_{end}}}{2}$$
(recall that $vol|B_{end}$ is continuous), and hence $L_x$ cannot be a non-trivial cover of $L_y$: in particular $L_x\cap N_{\alpha}(y')$ is a unique point for every $y'\in L_y$. Thus $L_x\cap N_{\frac{\alpha}{3}}(x')=\{x'\}$  for every $x'\in L_x$, proving the first part of the Lemma.  The second part follows from this.
\end{proof} 

From now on the numbers $\alpha$ and $\delta$ will be fixed and $\epsilon$ will be considered sufficiently small so the previous Lemma is valid. We then have the following Proposition (compare Lemma 3.4 in \cite{TeoErg} and Lemma 3.3 in \cite{ErgAnAc}).

\begin{proposition}\label{c0implicac1}
Given $\gamma>0$ there exists $\epsilon>0$ such that if  $x,y\in \Gamma(\delta)$ and $d_{\mathcal{C}^0}(L_x,L_y)<\epsilon$ then $d_{\mathcal{C}^1}(L_x,L_y)<\gamma$.
\end{proposition}

\begin{proof}

Observe that for every $x,y\in \Gamma(\delta)$ the leaves $L_x,L_y$ are homeomorphic. We fix $L$ the model of this homeomorphism class and consider a $\mathcal{C}^1$ embedding $h_x:L\rightarrow L_x$. Define $k:L\rightarrow L_y$ by $k(l)=N_{\alpha}(h_x(l))\cap L_y$: by the second part of the previous Lemma the map $k$ is well defined and one to one.

It follows that $k$ is a codimension zero submersion, hence an open embedding. As $L$\ is compact and connected, $k(L)=L_y$. Note that for every $z\in L_x$ there exists some vector $v_z\in T_zN, \norm{v_z}\leq \alpha$ such that
$$
N_{\alpha}(z)\cap L_y=\exp_z(v_z).
$$
Since the angle $\angle(N,E^c)$ is uniformly bounded from below and since $d_0\exp=Id$, by shrinking $\epsilon$ we get that $d_{\mathcal{C}^1}(h_x,k)<\gamma$, and this concludes the proof.
\end{proof}

\begin{remark} Note that in fact our proof gives that $k$ is isotopic to $h_x$, the isotopy being $h_t(z)=\exp_{h_x(z)}(tv_{h_x(z)})$.
\end{remark}

We now fix $0<\gamma_0<\frac{\alpha}{3}$ much smaller than the local product structure of \Fc, and for $0<\gamma\leq\gamma_0$ consider $0<\epsilon<\frac{\gamma}{2}$ satisfying the conclusion of the previous Proposition.

Suppose that $x,y\in \Gamma(\delta)$, $d_{\mathcal{C}^0}(L_x,L_y)<\epsilon$ and consider  
$\mathcal{C}^1$ embeddings $h_x,h_y:L\rightarrow M$ of $L_x,L_y$ such that $d_{\mathcal{C}^1}(h_x,h_y)<\gamma$. Observe that 
$$
\Wsl[\epsilon]{L_x}\subset \cup_{l\in L} N_{\frac{3\epsilon}{2}}(h_x(l)) \subset \Wsl[2\epsilon]{L_x}
$$
$$
\Wul[\epsilon]{L_y}\subset \cup_{l\in L} N_{\frac{3\epsilon}{2}}(h_y(l)) \subset \Wul[2\epsilon]{L_y}
$$
if $N$ is sufficiently close to $E^s\oplus E^u$ (and maybe shrinking $\gamma_0$). As the manifolds $\Wsl[\epsilon]{L_x},\Wul[\epsilon]{L_u}$ are transverse, we have that their intersection $F$ is a $c$-dimensional manifold tangent to $E^c$.

\begin{figure}[ht]
  \begin{center}
  \includegraphics[width=12cm]{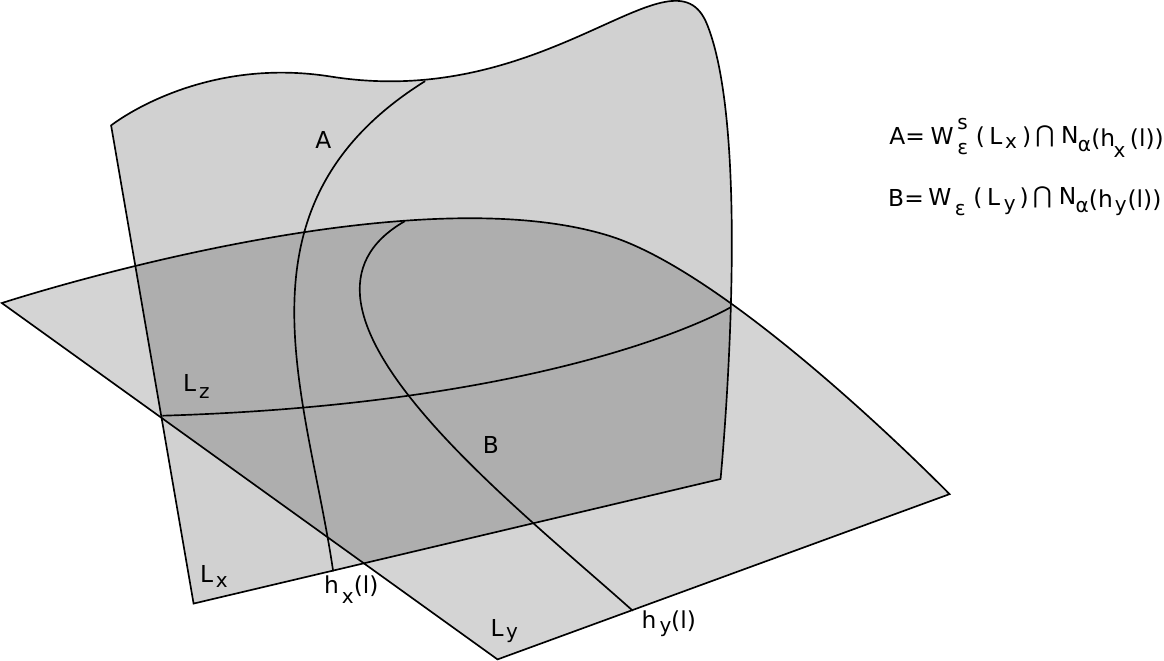}\\
  \caption{The leaf $L_z$.}\label{intersection}
\end{center}
\end{figure}

\emph{Claim}: For every $x'\in L_x$ the intersection $N_{\frac{3\epsilon}{2}}(x')\cap F$ consist of a unique point.

This is a direct consequence of Lemma \ref{unicainter}. We conclude then that $F$ consist of a single center leaf $L_z$, parametrized by the $\mathcal{C}^1$ embedding
$$
h_z(l):=N_{\frac{3\epsilon}{2}}(h_x(l))\cap F. 
$$
which is $\mathcal{C}^1-\gamma$ close to $h_x$. Likewise, the map
$$
h_z^0(l):=\Wsl[2\epsilon]{h_x(l)}\cap F 
$$
is a $\mathcal{C}^0$ parametrization of $L_z$, which is $\mathcal{C}^0-2\epsilon$ close to $h_x$. See figure \ref{intersection}. We have proved the following.

\begin{lemma}\label{coroGPS}
Given $0<\gamma\leq\gamma_0$ there exists $\epsilon>0$ such that if $x,y\in \Gamma(\delta)$, $d_{\mathcal{C}^0}(L_x,L_y)<\epsilon$. then there exists $z=z(x,y)$ such that the intersection $\Wsl[2\epsilon]{L_x}\cap \Wul[2\epsilon]{L_y}$ is a unique center leaf $L_z$. Furthermore,

\begin{enumerate}
\item $L_z$ is  $\mathcal{C}^{0}-2\epsilon$ close to $L_x,L_y$.
\item $L_z$ is  $\mathcal{C}^{1}-\gamma$ close to $L_x,L_y$.
\end{enumerate}
\end{lemma}

Ures' important observation is that for points $x,y\in B_0$ belonging to the same orbit the point $z$ inherits the same holonomy groups (stable and unstable), as explained next.

\begin{lemma}\label{mismaholo}
Given $0<\gamma<\gamma_0$ there exists $\epsilon>0$ such that if $x,y\in \Gamma(\delta)\cap B_0$ and
\begin{itemize}
\item $d_{\mathcal{C}^0}(L_x,L_y)<\epsilon$,
\item $y=f^n(x)$
\end{itemize}
then the point $z(x,y)$ given in the previous Lemma is also in $B_{0}$. Furthermore $|G^u_z|=\oo\Leftrightarrow |G^u_x|=\oo$ and $|G^s_z|=\oo\Leftrightarrow |G^s_x|=\oo$.
\end{lemma}

\begin{proof}
By dynamical coherence, and since we are assuming that $\gamma_0$ is much smaller than the local product structure constant of $\Fc$ we have that a small neighborhood $A\subset \Wu{L_x}$ of $L_x$ is mapped homeomorphically by the stable foliation to a neighborhood $A'\subset \Wu{L_z}$ of $L_z$. We conclude then that $G^u_x$ is infinite if and only if $G_z^u$ is infinite. Similarly, $G^s_y$ is infinite if and only if $G_z^s$ is infinite.
As $y=f^n(x)$ and $f$-preserves all foliations $\Fc,\Fcs,\Fcu$ we have that $G^s_y$ is isomorphic to $G^s_x$. By Proposition \ref{producto} at least one of the groups  $G^s_x,G^u_x$ is infinite, and thus again using the same Proposition we obtain that $G_z$ is infinite.
\end{proof}

We are ready now to prove the main result.

\begin{theoremsn}\label{A51}
If the Bad set $B_0$ of \Fc\ is non-empty, then there exist a $f$-periodic leaf contained in $B_0$.
\end{theoremsn}

\begin{proof}
Let $\lambda=\max\{\norm{df|E^s},\norm{df^{-1}|E^u}\}$. As we are looking for periodic leaves, by taking a power of $f$ it is not loss of generality to assume that $\lambda<\frac{1}{2}$. Fix $0<\gamma<\gamma_0$ and consider its corresponding $\epsilon>0$ with $2\epsilon(\frac{1}{1-\lambda}+1)<\epsilon_0$.  

Now take $\rho>0$ such that for $x,y \in \Gamma(\delta),\ d_{C_{0}}(L_x,L_y)<\rho $ it holds 
that for every $z\in \Gamma(\delta)\cap \Wsl[\lambda\epsilon]{L_y}$ the intersection  
$$
\Wsl[\epsilon]{L_x}\cap \Wul[\epsilon]{L_z}
$$
consist of a (unique) center leaf. This makes sense since $E^{cs}, E^{cu}$ are transverse, and since $\mathcal{C}^0$ close-by center leaves in $\Gamma(\delta)$ are $\mathcal{C}^1$ close, by election of $\epsilon$ (see Proposition \ref{c0implicac1} and Lemma \ref{coroGPS}).

The set $B_{end}$ is an invariant locally trivial compact fibration, hence we can find $n>0$ and $x,y=f^n(x)\in B_{end}$ such that $d_{\mathcal{C}^0}(L,f^nL)<\rho$, where $L=L_x$. By replacing $f$ with $f^{-1}$ it is also no loss of generality to assume that $|G_x^u|=\oo$, and thus by Lemma \ref{mismaholo} for the corresponding point $z(x,y)$ we have $|G_z^u|=\oo$. 

Fix an integer $k>0$ and consider the $\rho$-pseudo orbit of leaves $\{F_j\}_{j=1}^{k}$ with
$$
F_j=f^{j\, mod\, n}(L)\subset B_{end}.
$$
Note that the unstable holonomy group of $F_j$ is infinite. Now define recursively $\{F_j'\}_{j=1}^{k}$ as follows. Set $F'_0:=L$ and assume that we have already determined $F'_0,\ldots, F'_j$ with  $j<k$ satisfying 

\begin{enumerate}

\item $F'_j=\Wsl[\epsilon]{F_j}\cap \Wul[\epsilon]{fF_{j-1}'}$
\item $d_{\mathcal{C}^0}(F_j,F_j')<\epsilon$
\item $|G^{u}_{x_j'}|=\oo$, where $x_j'\in F_j'$.

\end{enumerate}

By composing the corresponding parametrizations with $f$ we deduce that  $d_{\mathcal{C}^0}(fF_j,fF_j')<\lambda\epsilon$, and since $d_{\mathcal{C}^0}(F_{j+1},fF_j)<\rho$ we conclude using Lemma \ref{coroGPS} that 
$$
F'_{j+1}=\Wsl[\epsilon]{F_{j+1}}\cap \Wul[\epsilon]{fF_j'}
$$
is a well defined center leaf; by Lemma \ref{mismaholo} it has infinite unstable holonomy group. This permits us to continue the induction and define $F_j'$ for $0,\ldots,k$.

Observe that for every $j$, $F_{j+1}'\subset\Wul[\epsilon]{fF_j'}$ and $d_{\mathcal{C}^0}(fF_j',F_{j+1}')<\epsilon$, hence  $f^{-1}(F_{j+1}')\subset \Wul[\lambda\epsilon]{F_j'}$ and $d_{\mathcal{C}^0}(F_j',f^{-1}(F_{j+1}'))<\lambda\epsilon$. We define $F'=f^{-k}(F_k')$ and conclude by recurrence that
$$
d_{\mathcal{C}^0}(f^jF',F_j')<\frac{\epsilon}{1-\lambda} \quad \forall\ 0\leq j\leq k,
$$
which in turn implies
$$
d_{\mathcal{C}^0}(f^jF',F_j)<\frac{\epsilon}{1-\lambda}+\epsilon<\frac{\epsilon_0}{2}.
$$

We now consider the bi-infinite sequence of leaves $\{F_j=f^{j\, mod\, n}(L)\}_{j\in \mathbb{Z}}$ and construct (by shifting) for every $k\geq 1$ a leaf $L_k\subset B_0$ such that 
$$
\forall -kn\leq j\leq kn,\quad  \dz{f^jL_k, F_j}\leq \frac{\epsilon_0}{2}.
$$
By compactness of $B_{0}$ the leaves $L_{k}$ have an accumulation point on a leaf $L'\subset B_{0}$. 

We claim that $L'$ is $\mathcal{C}^0-\epsilon_0$ close to $L$. It suffices to show that for every $x'\in L$  the intersection $L'\cap N_{\alpha}(x')$ consists of a unique point. If this were not the case, by continuous dependence on compact sets of the foliation \Fc\ we could find a leaf $L_k$ intersecting $N_{\alpha}(x')$ more than once, a contradiction to Lemma \ref{unicainter} since $L_k\subset \Gamma(\delta)$.  Likewise, 
$$
d_{\mathcal{C}^0}(f^jL',F_j)< \epsilon_0 \quad \forall j\in \mathbb{Z}. 
$$

We deduce that 
$$
d_{\mathcal{C}^0}(f^jL',f^j(f^nL'))< 2\epsilon_0 \quad \forall j\in \mathbb{Z}. 
$$
Apply parts (c) and (e) of Theorem 6.1 in \cite{HPS} with $V=\cup_{F\in\Fc} F, i:V\rightarrow M$ the natural inclusion to conclude that 
$$
f^nL'\subset\Wsl[2\epsilon_0]{L'}\cap \Wul[2\epsilon_0]{L'}
$$ 
and by (uniform) transversality of $E^{cs},E^{cu}$ the previous intersection consists only of the leaf $L'$ if $\epsilon_0$ is sufficiently small. Hence $f^nL'=L'$ and we are done.
\end{proof}

\begin{remark}
 With the same reasoning one can establish that sufficiently close pseudo-orbits of leaves $\{L_k\}_k$ in $B_{end}$ with isomorphic stable and unstable holonomy groups can be shadowed by a leaf in $B_0$. The notation becomes more cumbersome, and since we will not use this fact in this article, we have opted to present the proof of the restricted version above.
\end{remark}

\section{Concluding Remarks}

The fact of the center foliation of a partially hyperbolic diffeomorphism being uniformly compact  has implications both on the geometry of the foliation and on the dynamics of the corresponding map. We give here a brief discussion of some of these consequences.

A  partially hyperbolic diffeomorphism  $f$\ is said to be plaque expansive if there exists $\epsilon>0$\ such that if $\{x_n\}_{n\in\mathbb{Z}},\{y_n\}_{n\in\mathbb{Z}}$\ are two pseudo-orbits subordinated to the center foliation\footnote{That is $x_{n+1}\in \Wc{f(x_n)}$\ for every $n\in\mathbb{Z}$\ (and likewise for $\{y_n\}_{n\in\mathbb{Z}})$.} satisfying for every $n\in\mathbb{Z}$\
$$
d(x_n,y_n)<\epsilon
$$

then $x_0\in\Wc{y_0}$. As explained in \cite{HPS}, the fact that $f$\ is plaque expansive implies that the pair $(f,\Fc)$\ is stable in the following sense: there exist a $\mathcal{C}^1$-neighborhood $U$\ of $f$\ such that if $g\in U$\ then
\begin{enumerate}
\item $g$\ is partially hyperbolic, dynamically coherent and plaque expansive.
\item There exists an homeomorphism $h:M\rightarrow M$\ close to the identity such that $h^{\ast}(\Fc_f)=\Fc_g$.
\end{enumerate}

In the case when \Fc\ is uniformly compact it is not too hard to prove that $f$ is plaque expansive. This is achieved by applying Reeb's Stability Theorem as in the proof of Proposition \ref{unifimpcomp}. See 
Proposition 13 in \cite{HolFolRev} for the complete argument, or Section 6 in \cite{CompCenFolFinHol} for a different proof\footnote{P. Berger pointed me out that this result also follows from
a modification of his results obtained in Appendix C of \cite{persistence}.}.  Hence:

\begin{theorem}
If \Fc\ is uniformly compact then $(f,\Fc)$ is stable.
\end{theorem}

In fact the pairs $(f,\Fcs),(f,\Fcu)$\ are also stable. See Section 4 in \cite{CompDynFol}.

\begin{definition}
A foliation $\mathcal{F}$\ on a manifold $M$ is said to be of \emph{uniform type} if all leaves of $\mathcal{F}$\ have homeomorphic universal covers.
\end{definition}

For example if $(\phi_t)_t$\ is a flow on $M$ without singularities,  then the foliation induced by the orbits of $(\phi_t)_t$\ is of uniform type. More generally, if $G$\ is a Lie group acting on $M$ and the action is effective and locally free then the orbit-foliation is of uniform type.

\begin{proposition}\
Let $f:M\rightarrow M$ be a partially hyperbolic diffeomorphism with uniformly compact center foliation \Fc and $L\in\Fc$. Then the foliations $\Fc|\Ws{L}$ and $\Fc|\Wc{L}$ are 
of uniform type.
\end{proposition}

\begin{proof}
This follows from the proof of Proposition \ref{unifimpcomp}: if $L'\in \Ws{L}$, then by some positive iterate $n$ both leaves $f^n(L)$\ and $f^n(L')$\ are in the same Reeb neighborhood,
where by the discussion of Section 2 all leaves have the same universal covering. Apply $f^{-n}$ to obtain the claim. Similarly for $L'\in \Wu{L}$.
\end{proof}

\begin{corollary}
Let $f:M\rightarrow M$ be a partially hyperbolic diffeomorphism with uniformly compact center foliation \Fc\ and assume either that
\begin{enumerate}
 \item $f$ is \emph{centrally transitive}, meaning that there exists a center leaf $L$ whose orbit is dense, or
 \item $f$ is \emph{accessible}, meaning that given $x,y\in M$\ there exists a piecewise $\mathcal{C}^1$ curve $c:[0,1]\rightarrow M$\ whose tangent is always contained in $E^s$\ or $E^u$\ and such that $c(0)=x,c(1)=y$.
\end{enumerate}
Then \Fc\ is of uniform type.
\end{corollary}
The proof is clear.

To conclude, we consider a  partially hyperbolic diffeomorphism $f:M\rightarrow M$  with uniformly compact center foliation \Fc, and denote by $X=M/\Fc$ and $g:X\rightarrow X$  
the map induced by $f$. It is proven in \cite{CompCenFolFinHol} (Theorem 2) that $g$ has the \emph{pseudo-orbit tracing property}, meaning that any sufficiently small $g$-pseudo-orbit can be
shadowed by a true orbit, although the orbit may not be unique (cf. the example after Proposition \ref{producto}).  An improvement of this result can be obtained in the following 
case:

\begin{theorem}
Assume furthermore that $E^u$ is one dimensional and that the center foliation is without holonomy. Then $X$ is a $c$-dimensional torus and $g$ is conjugate to a linear
Anosov diffeomorphism.
\end{theorem}

This Theorem was obtained by A. Gogolev \cite{gogo} and generalized by D. Bohnet \cite{CompCenFolFinHol2}, who removed the condition of trivial holonomy.

\vspace{0.2cm}

\textbf{Acknowledgments}:  This article contains results which I obtained in my thesis \cite{CompDynFol} under the supervision of Charles Pugh and Michael Shub. I express here my deep gratitude to them, for everything that they taught me and all the support that they gave during these years. I also want to thank Federico Rodriguez-Hertz and Ra\'ul Ures for their help in the preparation of this article. I am indebted to A. Gogolev who patiently pointed me out the necessity of more careful arguments than the ones appearing in the first versions of this work, and to Christian Bonatti who besides of finding mistakes in my previous versions gave me several suggestions which significantly contributed to the improvement of the paper. Finally, I would like to thank Enrique Pujals for helpful discussions, for reading the drafts and specially for all his support.

The author was partially funded by a C.N.P.Q. grant.

\section*{Appendix: Reeb's Stability Theorem}

Let $\F$ be a $\mathcal{C}^{1,0}$ codimension $q$ foliation on a manifold $M$, and $E$ a locally compact saturated set of $M$. We use the notation explained in Section 2, and in particular for a leaf $L_x$ we denote by $p_x:W_x\rightarrow L_x$ the differentiable bundle of $q$-open discs $D(y)=p^{-1}_x(y)$. Here we discuss the following result.

\begin{theorem}[Reeb's Stability Theorem]
Assume that $x\in E$ is such that its leaf $L_x$ is compact and $G_x|E$ is finite. Then there exist arbitrarily small relatively open sets $U\subset E$ such that
\begin{enumerate}
\item $U$ is saturated, and if $V:=U\cap D(x)$ there exists a finite covering $\widehat{L}$ of $L_x$ such that $U$ is homeomorphic to a suspension $S=\widehat{L}\times_{G_x|E} V$ by a homeomorphism which sends the leaves of $\F|U$ to the leaves of $S$.
\item For every $y\in U$ the map $p|:L_y\rightarrow L_x$ is a covering map with less than equal $|G_x|E|$ number of sheets. In particular for every $y\in U$ the group $G_y|E$ is finite.
\end{enumerate}
\end{theorem}

The version given above follows from the general version due to C. Ehresmann and W. Shih \cite{reebpotente}. Nonetheless, for the context discussed in this article (and probably for general\emph{ normally hyperbolic foliations} \cite{HPS}) the restricted case is enough. We will sketch the proof of the result since its derivation from Ehresmann-Shih article is not very direct. The arguments are an adaptation of Proposition 2.8 in \cite{BadSets2}. We stress here that the proof depends on the existence of a tubular neighborhoods for the leaf $L_x$. For a subset $A\subset M$ denote its saturation by $sat(A)$.

\begin{lemma}[Prop 4.1, III in \cite{FolCpctForm}]
In the hypothesis above there exist arbitrarily small neighborhoods $V\subset D_x$ where the holonomy group is represented by a group $H_V$ of homeomorphisms of $V$.
\end{lemma}

\begin{proof}
The fact that $L_x$ has arbitrarily small saturated sets in $E$ is classical (cf. Theorem 4.2 in \cite{FolCpct}), and depends on the fact that $E$ is locally compact. Consider a finite family of local homeomorphism $\{h_i:V_i\subset D(x)\rightarrow D(x)\}$ whose germs represent the holonomy group $G_x|E$.  For each pair on indices $i,j$ let $k=k(i,j)$ be the unique index such that $germ_x(h_k)=germ_x(h_j)\circ germ_x(h_i)$. Since we are only interested in the germs, it is not loss of generality to assume that
$$
\forall i,j,\ h_j\circ h_i|(V_i\cap h_i^{-1}V_j\cap V_k)=h_k|(V_i\cap h_i^{-1}V_j\cap V_k).
$$
Set $V:=\cap_{i,j}V_i\cap h_i^{-1}V_j, g_i:=h_i$ and observe that its saturation $U=sat(V)\cap E$ in $E$ is invariant under the group $H_V=\{g_i\}$. From this we deduce that $H_V$ is isomorphic to $G_x|E$. Clearly $V$ can be taken arbitrarily small.
\end{proof}

\begin{proof}[Reeb's Stability Theorem]
Fix $V$ as in the previous Lemma  and set $U:=sat(V)\cap E$. As explained in Section 2, the holonomy representation splits as
$$
\pi_1(L_x,x)\xrightarrow[]{\psi}H_V\subset Homeo_x(V)\rightarrow G_x|E,
$$
and if we denote by $\rho:(\widehat{L},\widehat{x})\rightarrow (L_x,x)$ the covering corresponding to $\ker{\psi}$, then $H_V$ is isomorphic to its deck transformation group. For $\alpha\in \pi_1(L_x,x)$ let $g_{\alpha}\in H_V$ the holonomy map that it determines, and for $[\alpha]\in \pi_1(L,x)/\ker(\psi)=H_V$ let
$T_{[\alpha]}$ be the covering transformation of $\widehat{L}$ that it represents. The group $H_V$ acts on the space $X=\widehat{L}\times V$ by:
$$
[\alpha]\in H_V , l\in \widehat{L},v\in V\Rightarrow  [\alpha]\cdot (l,v)= (T_{[\alpha]}^{-1}(l),g_{\alpha}(v)).
$$
The action is well defined since if $[\alpha]=[\beta]\in H_V$ then $\alpha\beta^{-1}\in \ker(\psi)$ and hence $g_{\alpha}(v)=g_{\beta}(v)$. Now if $(l,v)\in X$ we take a path $\alpha_l:[0,1]\rightarrow\widehat{L}$ such that $\alpha_l(0)=\widehat{x}, \alpha_l(1)=l$. Then $\rho(\alpha_l)$ is a path in $L$ and we lift this path using $p$ to  a path $\beta_{l,v}$ in $L_v$ with starting point $v$.

\emph{Claim}: If $\alpha'_l$ is any other path in $\widehat{L}$ joining $\widehat{x}$ and $l$, and we denote by $\beta'_{l,v}$ the corresponding path in $L_v$ obtained by the previous procedure, then $\beta_{l,v}(1)=\beta'_{l,v}(1)$.

This follows since $\alpha'_l\alpha^{-1}_l$ is a loop at $\widehat{x}$, hence $\rho(\alpha'_l\alpha^{-1})\in \ker{\psi}$.

We now define the map $k:X\rightarrow U$ by $k(l,v)=\beta_{l,v}(1)$. Using that $L$ is a nice space (in particular locally simply connected) one verifies that $k$ is continuous, and with not too much effort one can show that $k$ is surjective. Observe that for $y,z\in L_x$ the holonomy transport between $D(y)$ and $D(z)$ is an open map, and this implies that $k$ is open. Finally note that $k$ is constant precisely on orbits of the action of $H_V$, hence it induces an homeomorphism $K:\widehat{L}\times_{H_V} V\rightarrow U$. The rest of the claims follows from this.
\end{proof}

\bibliographystyle{alpha}
\bibliography{biblio}
\end{document}